\documentclass[oneside,a4paper,12pt,notitlepage]{article}
\usepackage[left=0.8in, right=0.8in, top=1.5in, bottom=1.5in]{geometry}
\usepackage[T1]{fontenc} 
\usepackage[utf8]{inputenc} 
\usepackage{lipsum} 
\usepackage{lmodern}
\usepackage{amssymb}
\usepackage{amsthm}
\usepackage{bm}
\usepackage{mathtools}
\usepackage{braket}
\usepackage{esint}
\newcommand{\abs}[1]{{\left|#1\right|}}

\usepackage{booktabs}
\usepackage{graphicx}
\usepackage{tikz}
\usetikzlibrary{patterns}
\usepackage{multicol}
\usepackage{caption}
\usepackage{enumerate}
\usepackage[skins,theorems]{tcolorbox}
\tcbset{highlight math style={enhanced,
		colframe=black,colback=white,arc=0pt,boxrule=1pt}}
\def\XXint#1#2#3{{\setbox0=\hbox{$#1{#2#3}{\int}$}
    \vcenter{\hbox{$#2#3$}}\kern-.5\wd0}}

\theoremstyle{definition}

\theoremstyle{plain}
\newtheorem{teorema}{Theorem}[section]

\newtheorem{lemma}[teorema]{Lemma}
\newtheorem{prop}[teorema]{Proposition}

\theoremstyle{definition}

\renewcommand{\div}{\text{div}}

\DeclareMathOperator{\R}{\mathbb{R}}

\DeclareMathOperator{\supp}{supp}	

\makeatletter
\newcommand{\myfootnote}[2]{\begingroup
	\def\@makefnmark{}%
	\addtocounter{footnote}{-1}%
	\footnote{\textbf{#1} #2}%
	\endgroup}
\makeatother
\usepackage{hyperref}
\hypersetup{linktoc=none, bookmarksnumbered, colorlinks=true, linkcolor=red}

\title{Shape optimization for a nonlinear elliptic problem related to thermal insulation}
\author{Rosa Barbato}
\date{\today}
\begin{document}
	\maketitle
	\myfootnote{}{2010 Mathematics Subject Classification: 35J66, 49Q10. }
	\myfootnote{}{Key words and phrases: $p$-Laplacian, Robin boundary conditions.}
    \vspace{-0.8cm}
	\begin{abstract}
 In this paper we consider a minimization problem of the type
 $$	I_{\beta,p}(D;\Omega)=\inf\biggl\{\int_\Omega \abs{D\phi}^pdx+\beta \int_{\partial^* \Omega}\abs{\phi}^pd\mathcal{H}^{n-1},\; \phi \in W^{1,p}(\Omega),\;\phi \geq 1 \;\textrm{in}\;D\biggl\},$$
 where $\Omega$ is a bounded connected open set in $\R^n$, $D\subset \bar{\Omega}$ is a compact set and $\beta$ is a positive constant. 
 \\ We let the set $D$ vary under prescribed geometrical constraints and $\Omega \setminus D$ of fixed thickness, in order to look for the best (or worst) geometry in terms of  minimization (or maximization) of $I_{\beta,p}$. In the planar case, we show that under perimeter constraint the disk maximize $I_{\beta,p}$.\\ In the $n$-dimensional case we restrict our analysis to convex sets showing that the same is true for the ball but under different geometrical constraints.
	
	\end{abstract}
\section{Introduction}
Insulation problems have interested many researchers and it's a very active field of research as it is related to environmental improvement. In fact, thermal insulation may be applied to save energy and even if it may seem counterintuitive, having too much insulation can have a negative impact on the energy efficiency; furthermore, adding too much insulation is expensive and unsustainable.\\
In this paper we consider a problem of this type: let $\Omega$ be a bounded connected open set of $\R^n$,  $D\subset \overline{\Omega}$ a compact set, $\beta>0$ a fixed constant.
\begin{equation}
	\label{prob}
	I_{\beta,p}(D;\Omega)=\inf\biggl\{\int_\Omega \abs{D\phi}^pdx+\beta \int_{\partial \Omega}\abs{\phi}^pd\mathcal{H}^{n-1},\; \phi \in W^{1,p}(\Omega),\;\phi \geq 1 \;\textrm{in}\;D\biggl\}.
\end{equation}
  Our aim is to study maximization and minimization problems of $I_{\beta,p}(D,\Omega)$, among domains with given geometrical constraints. 
 If $\Omega$ has Lipschitz boundary, then there exists a minimizer $u$ of (\ref{prob}) that satisfies 
	$$
	\begin{cases}
		\Delta_p u=0\qquad &\textrm{in}\; \Omega \setminus \overline{D}\\
		u=1 \qquad & \textrm{in}\; D\\
		{\abs{Du}}^{p-2}\dfrac{\partial u}{\partial \nu}+\beta {\abs{u}}^{p-2}u=0 \qquad & \textrm{on}\; \partial \Omega
	\end{cases}
	$$  

where $$\Delta_p u= \div(\abs{Du}^{p-2}Du)$$ is the $p$-Laplace operator, and the functional $I_{\beta, p}(\Omega, D)$ assumes the following form $$I_{\beta,p}(\Omega, D)=\beta\int_{\partial\Omega}{\abs{u}}^{p-2}u\; d\mathcal{H}^{n-1}.$$
In this order of ideas, the case $p=2$ has been treated in \cite{prof}. In this case the minimization problem arises from a thermal insulation problem. Related results are also contained in \cite{butt, caff, bucur, nahon, simm, nitsch, prof, window, esposito, frie}.\\ In the present paper we consider the general case $1<p< +\infty$.\\ The plan of the paper is the following: after recalling some well-known fact on convex domains in section \ref{sec2}, we prove some basic properties of $I_{\beta,p}(\Omega,D)$ in Section \ref{sec3}.\\
In section \ref{sez4}-\ref{sec6} we discuss our main results. First, we consider the domains $\Omega=D+\delta B$, with $\delta>0$; in the planar case studied in section \ref{sez4}, we found that if $D$ is an open, bounded, connected set of $\R^n$ with piecewise $C^1$ boundary, the maximum of $I_{\beta,p}(D, D+\delta B)$ is achieved at the disk having the same perimeter of $D$. In section \ref{sez5} we consider the $n$-dimensional case for convex sets and we obtain that among the convex sets $\Omega=D+\delta B$, with fixed $\delta$ and of given $W_{n-1}$ quermassintegral of $D$, the maximum is attained when $D$ is a ball.\\
Finally in section \ref{sec6}, we show a counterintuitive behaviour of the functional $I_{\beta,p}(D, \Omega)$; indeed we prove that for suitable values of $\beta$, there exists a positive constant $\delta_0$ such that for any bounded domain $\Omega$, with $D\subset \Omega$ and $\abs{\Omega}-\abs{D}<\delta_0$, then $I_{\beta,p}(B_R,\Omega)>I_{\beta,p}(B_R,B_R)$.
\section{Preliminaries}
\label{sec2}
Here we list some basic facts on convex sets (see, for example \cite{burago, convex}). Let $K$ be a nonempty, bounded, convex set in $\R^n$ and let $\delta>0$. Then the Steiner formulas for the volume and the perimeter read as

	$$
	\begin{aligned}
		\abs{K+\delta B}&=\sum_{j=0}^{n}\binom{n}{j} W_j(K)\delta^j\\
		&=\abs{K}+nW_1(K)\delta+\dfrac{n(n-1)}{2}W_2(K)\delta^2+\cdots+\omega_n \delta^n
	\end{aligned}
$$
and

\begin{equation}
	\label{steinerconvex}
	\begin{aligned}
		P(K+\delta B)&=n \sum_{j=0}^{n-1}\binom{n-1}{j}W_{j+1}(K)\delta^j\\[5pt]
		&=P(K)+n(n-1)W_2(K)\delta +\cdots +n\omega_n \delta^{n-1},
	\end{aligned}
\end{equation}

 where $B$ is the unit ball in $\R^n$ centered at the origin, whose measure is denoted by $\omega_n$  and $K+\delta B$ stands for the Minkowski sum.\\The coefficients $W_j(K)$ are the so-called quermassintegrals of $K$. In particular $W_0$ is the volume of $K$, $W_1=\frac{P}{n}$ and $W_n=\omega_n$.\\
 It immediately follows that 
 \begin{equation}
 	\label{12}
 	\lim\limits_{\delta \rightarrow 0^+}\dfrac{P(K+\delta B)-P(K)}{\delta}=n(n-1)W_2(K).
 \end{equation}

 If $K$ has $C^2$ boundary, with nonzero Gaussian curvature, the quermassintegrals are related to the principal curvatures of $\partial K$. Indeed, in such a case
 \begin{equation}
 	\label{13}
 	W_i(K)=\frac{1}{n}\int_{\partial K} H_{i-1}(x)d\mathcal{H}^{n-1}, \qquad i=1,\cdots ,n.
 \end{equation} 
 Here $H_j$ denotes the $j$-th normalized elementary symmetric function of the principal curvatures of $\partial K$, that is $H_0=1$, and 
 $$H_j(x)={\binom{n-1}{j} }^{-1}\sum_{1\leq i_1 \leq \cdots \leq i_j\leq n-1}{k_{i_1}(x)\cdots k_{i_j}(x)}, \qquad j=1, \cdots, n-1,$$ 
 where $k_{1}(x)\cdots k_{n-1}(x)$ are the principal curvatures at a point $x \in \partial K$. In particular, by (\ref{12}) and (\ref{13}) we get also that 
 $$\lim\limits_{\delta \rightarrow 0^+}\dfrac{P(K+\delta B)-P(K)}{\delta}=(n-1)\int_{\partial K}H_1(x)d\mathcal{H}^{n-1},$$ where $H_1(x)$ is the mean curvature of $\partial K$ at a point $x$.\\ The Alexandrov-Fenchel inequalities state that
 \begin{equation}
 	\label{14}
 	\biggl(\dfrac{W_j(K)}{\omega_n}\biggl)^{\frac{1}{n-j}}\geq \biggl(\dfrac{W_i(K)}{\omega_n}\biggl)^{\frac{1}{n-i}}, \qquad0\leq i < j \leq n-1,
 \end{equation} 

 where the inequality is replaced by an equality if and only if $K$ is a ball. \\ In what follows, the Alexandrov-Fenchel inequalities will be used for particular values of $i$ and $j$. When $i=0$ and $j=1$, we have the classical isoperimetric inequality: 
 $$P(K)\geq n {\omega_n}^{\frac{1}{n}}{\abs{K}}^{1-\frac{1}{n}}.$$
 Moreover, if $i=k-1,$ and $j=k$, we have 
 $$W_k(K)\geq {\omega_n}^{\frac{1}{n-k+1}} W_{k-1}(K)^{\frac{n-k}{n-k+1}}.$$
 Let us denote by $K^*$ a ball such that $W_{n-1}(K)=W_{n-1}(K^*).$ Then by the Alexandrov-Fenchel inequalities (\ref{14}), for $0 \leq i < n-1$ it holds that
 $$\biggl(\dfrac{W_i(K^*)}{\omega_n}\biggl)^{\frac{1}{n-i}}=\dfrac{W_{n-1}(K^*)}{\omega_n}=\dfrac{W_{n-1}(K)}{\omega_n}\geq \biggl(\dfrac{W_i(K)}{\omega_n}\biggl)^{\frac{1}{n-i}}, $$
 hence
\begin{equation}
	\label{quer}
	 W_i(K)\leq W_i(K^*), \qquad 0\leq i\leq n-1.
\end{equation}
 
 Now consider the the two dimensional case. For an open, bounded, connected set $D \subset \R^2$ we denote by $D^*$ the disk having the same perimeter of $D$. If $\Omega=D+\delta B$ and $\Omega_*=D^*+\delta_* B$, where $B$ is the disk centered at the origin, the Steiner formulae become $$\abs{\Omega}=\abs{D}+P(D)\delta + \pi \delta ^2, \qquad P(\Omega)=P(D)+2\pi \delta$$
 $$\abs{\Omega_*}=\abs{D^*}+P(D^*)\delta_* + \pi \delta_* ^2, \qquad P(\Omega_*)=P(D^*)+2\pi \delta_*.$$
 Let us observe that, in our context, if we ask that the area of the insulating material $\Omega \setminus \overline{D}$ remains constant, then 
 $$\abs{\Omega}-\abs{D}=P(D)\delta +\pi \delta ^2=\abs{\Omega_*}-\abs{D^*}=P(D^*)\delta_*+\pi \delta_*^2.$$
 Since $P(D)=P(D^*)$, then $\delta=\delta_*$ and, as byproduct, $P(\Omega)=P(\Omega_*)$. On the contrary, if $\delta=\delta_*$, then $\abs{\Omega}-\abs{D}=\abs{\Omega_*}-\abs{D^*}$. \\
 For a general bounded domain with piecewise $C^1$ boundary, it holds that
 \begin{equation}
 	\label{steiner}
 	\abs{\Omega}\leq\abs{D}+P(D)\delta + \pi \delta ^2, \qquad P(\Omega) \leq P(D)+2\pi \delta,
 \end{equation}
 hence $\delta = \delta_*$ implies 
 $$\abs{\Omega}-\abs{D}\leq \abs{\Omega_*}-\abs{D^*},$$
 which means that the area of the insulating material increases keeping fixed the perimeter $P(D)$ and the thickness $\delta$.
 
 \section{The variational problem}
 \label{sec3}
Given a compact set $D$ in $\R^n$ and a bounded connected open set $\Omega$ of $\R^n$ with $D \subset \bar{\Omega}$, we are interested to study 
\begin{equation}
	\label{funzionale}
	I_{\beta,p}(\Omega, D)=\inf\left\{\int_\Omega \abs{D\phi}^{p}dx+\beta\int_{\partial^{*}\Omega}\abs{\phi}^{p}d\mathcal{H}^{n-1},\; \phi\in W^{1,p}(\Omega),\; \phi \geq 1 \;\text{in}\; \overline{D}\right\}, 
\end{equation}
with $\beta>0$. The following result holds.
\begin{prop}
	If $\Omega$ is bounded connected open set of $\R^n$ with Lipschitz boundary, and $D$ is a compact set with $D\subset \Omega$, then there exists a unique positive minimizer $u \in W^{1,p}(\Omega)$ of (\ref{funzionale}) such that satisfies
		
\begin{equation}
	\label{theproblem}
		\begin{cases}
		\Delta_p u=0\qquad &\textrm{in}\; \Omega \setminus \overline{D}\\
		u=1 \qquad & \textrm{in}\; D\\
		{\abs{Du}}^{p-2}\dfrac{\partial u}{\partial \nu}+\beta {\abs{u}}^{p-2}u=0 \qquad & \textrm{on}\; \partial \Omega
	\end{cases}
\end{equation}
	
Moreover the functional can be written as 
	\begin{equation}
		\label{8bis}
		I_{\beta,p}(\Omega, D)=\beta\int_{\partial\Omega}{u}^{p-1}\; d\mathcal{H}^{n-1},
\end{equation}
where $u$ is the solution to (\ref{theproblem}).
\end{prop}
\begin{proof}
	The proof follows by a standard application of a calculus of variation's argument. We sketch the proof for completeness.\\
	Let $u_n$ be a minimizing sequence for $I_{\beta,p}(\Omega, D)$. Then $u_n$ is bounded in $W^{1,p}$, $\exists u_{n_k} \rightarrow u$ weakly in $W^{1,p}(\Omega)$, strongly in $L^p$ and a.e in $\Omega$. Then by semicontinuity, $u$ is a minimizer of $I_{\beta,p}(\Omega, D)$. Moreover, being $\abs{u}$ still a minimizer, we can assume $u \geq 0$. By Harnack, $u>0$ in $\Omega$. Furthermore, by convexity, $u$ is the unique minimizer and it satisfies 
	$$
	\begin{cases}
		\Delta_p u=0\qquad &\textrm{in}\; \Omega\\
		u=1 \qquad & \textrm{in}\; \bar{D}\\
		{\abs{Du}}^{p-2}\dfrac{\partial u}{\partial \nu}+\beta {u}^{p-1}=0 \qquad & \textrm{on}\; \partial \Omega
	\end{cases}
	$$  
If we denote with $W^{1,p}_{\overline{D}}(\Omega)$ the closure in $W^{1,p}(\Omega)$ of $\{\phi|_{\Omega}: \phi \in C^{\infty}(\R^n) \;\textrm{with}\; \bar{D} \cap \supp\phi=\emptyset\}$, this is true, by definition, if and only if 
	$$\int_{\Omega\setminus \overline{D}}\abs{Du}^{p-2}Du D\phi\;dx+\beta\int_{\partial \Omega}u^{p-1}\phi\;d\mathcal{H}^{n-1}=0$$
	for any $\phi \in W^{1,p}_{\overline{D}}(\Omega)$.\\
	So, if we take $\phi=u-1\in W^{1,p}_{\overline{D}}(\Omega)$,
	$$\int_\Omega {\abs{Du}}^{p-2}Du(Du)\;dx+\beta \int_{\partial\Omega}u^{p-1}u(u-1)\;d\mathcal{H}^{n-1}=0$$
	that is 
	$$\int_\Omega (Du)^{p}\;dx+\beta \int_{\partial\Omega}({u}^{p}-{u}^{p-1})\;d\mathcal{H}^{n-1}=0$$
	
	and this implies (\ref{8bis}).
\end{proof}
\section{The planar case}
\label{sez4}
In this section we consider the case where $\Omega$ is the Minkowski sum $$\Omega=D+\delta B,$$ $B$ is the unit ball centered at the origin and $\delta$ is a positive constant, representing the thickness of $\Omega\setminus D$. Then we set $$I_{\beta,p,\delta}(D)=I_{\beta,p}(D,D+\delta B).$$
\begin{teorema}
	Let $D$ be an open, bounded connected set of $\R^{2}$ with piecewise $C^{1}$ boundary. Then $$I_{\beta,p,\delta}(D)\leq I_{\beta,p,\delta}(D^*),$$
	where $D^*$ is the disk having the same perimeter of $D$.
\end{teorema}
\begin{proof}
	Let $v$ be the radial minimizer of $I_{\beta, p,\delta}(D^*)$ and $\Omega_*=D^*+\delta B$. Given $R$, the radius of $D^*$, then denote by $$v_m=v(R+\delta)=\min_{\Omega_*}v$$ and
	$$\max_{\Omega_*}v=v(R)=1.$$ Being $v$ radial, the modulus of the gradient of $v$ is constant on the level lines of $v$.\\ Let us consider the function $$g(t)=\abs{Dv}_{v=t},\qquad v_m<t\leq1$$
	and
	$$w(x)=G(R+d(x)),\;\; x\in\Omega,\qquad \textrm{where}\;\; G^{-1}(t)=R+\int_{t}^{1}\dfrac{1}{g(s)}ds$$
	and $d(x)$ is the distance of a point $x$ from $D$. In particular, $v$ is decreasing and the function $g$ can be zero only for $v=1$. This means that $G^{-1}$ is decreasing, therefore so is $G$ and then $w\in W^{1,p}(\Omega)$. It results:
		$$
		\begin{aligned}
			& \max_{\Omega}w=w\rvert_{\partial D}=1=G(R),\\
			& w_m=\min_{\Omega}w=w\rvert_{\partial\Omega}=G(R+\delta)=v_m,\\
			&\abs{Dw}_{w=t}=\abs{Dv}_{v=t}=g(t),\;\;w_m\leq t\leq1.
		\end{aligned}
	$$

Hence $w$ is a test function. Then $$I_{\beta,p,\delta}(D)\leq\int_{\Omega\setminus D}{\abs{Dw}}^p \;dx+\beta \int_{\partial\Omega}{\abs{w}}^p\;d\mathcal{H}^{1}.$$
Let $$E_t=\{x\in \Omega:w(x)>t\}=\{x\in \Omega:d(x)<G^{-1}(t)\}=D+G^{-1}(t)B$$
and let
$$B_t=\{x\in\Omega_*:v(x)>t\}.$$
By Steiner formula (\ref{steiner}) we get
$$P(\Omega)\leq P(D)+2\pi\delta,$$
so
$$P(E_t)\leq P(D)+2\pi G^{-1}(t)=P(D^*)+2\pi G^{-1}(t)=2\pi (R+G^{-1}(t))=P(B_t)$$
for every $t \in\; ]w_m,1]$.
Hence,
\begin{center}
	
	$$\int_{w=t} \abs{Dw}d\mathcal{H}^1=\int_{w=t}g(t)d\mathcal{H}^1
	=g(t)P(E_t)\leq g(t)P(B_t)=\int_{v=t}\abs{Dv}d\mathcal{H}^1,\qquad w_m<t\leq1.$$
	
\end{center}
Then, by co-area formula

\begin{equation}
	\label{eq1}
	\begin{aligned}
		\int_{\Omega\setminus \overline{D}}{\abs{Dw}}^{p}dx&=\int_{w_m}^{1}dt\int_{w=t}{\abs{Dw}}^{p-1}d\mathcal{H}^1\\
		&=\int_{w_m}^{1}[g(t)]^{p-1}P(E_t)dt\leq\int_{w_m}^{1}[g(t)]^{p-1}P(B_t)dt\\
		&=\int_{\Omega\setminus \overline{D}_*}{\abs{Dv}}^p dx.
	\end{aligned}
\end{equation}
Since by construction $w=w_m=v_m$ on $\partial \Omega$ and $P(\Omega)=P(\Omega_*)$, we have
\begin{equation}
	\label{eq2}
	\int_{\partial\Omega} {\abs{w}}^p d\mathcal{H}^1={\abs{w_m}}^{p}P(\Omega)={\abs{v_m}}^pP(\Omega_*)=\int_{\partial \Omega_*}{\abs{v}}^p d\mathcal{H}^1.
\end{equation}

Hence, by (\ref{eq1}) and (\ref{eq2}) it holds that
\begin{center}
	$$
	\begin{aligned}
		I_{\beta,p, \delta}(D)\leq\int_{\Omega\setminus D}{\abs{Dw}}^pdx+\beta\int_{\partial \Omega} {\abs{w}}^pd\mathcal{H}^1
		&\leq \int_{\Omega_*\setminus D^*}{\abs{Dv}}^pdx+\beta\int_{\partial \Omega_*}{\abs{v}}^pd\mathcal{H}^1=I_{\beta,p, \delta}(D^*).
	\end{aligned}
	$$
\end{center} 

\end{proof}
 \section{The $n$-dimensional case}
 \label{sez5}
 Now we prove that in higher dimension ($n \geq 3$) balls still maximize $I_{\beta, p, \delta}$, but our result finds its natural generalization in the class of convex domains $D$.
 \begin{teorema}
 	Let $D$ be an open, bounded, convex set of $\R^n$. Then $$I_{\beta,p,\delta}(D)\leq I_{\beta,p,\delta}(D^*),$$
 	where $D^*$ is the ball having the same $W_{n-1}$ quermassintegral of $D$, that is $W_{n-1}(D)=W_{n-1}(D^*)$.
 \end{teorema}
\begin{proof}
	Let be $\Omega_*=D^*+\delta B$, and $v$ the radial minimizer of $I_{\beta,p,\delta}(D^*)$. Since $\Omega=D+\delta B$, Steiner formula for the perimeter (\ref{steinerconvex}) and Aleksandrov-Fenchel inequalities (\ref{quer}) imply $P(\Omega)\leq P(\Omega^*)$.\\
	We denote by $v_m=v(R+\delta)=\min_{\Omega^*}v$ and by $\max_{\Omega_*}v=v(R)=1$. Being $v$ radial, the modulus of the gradient of $v$ is constant on the level lines of $v$.\\ Let us consider the function $$g(t)=\abs{Dv}_{v=t},\qquad v_m<t\leq1$$
	and
	$$w(x)=G(R+d(x)),\;\; x\in\Omega,\qquad \textrm{where}\;\; G^{-1}(t)=R+\int_{t}^{1}\dfrac{1}{g(s)}ds,$$
	and $d(x)$ is the distance of a point $x$ from $D$.\\ By construction $w\in 	W^{1,p}(\Omega)$ and being $G$ decreasing it results:
	\begin{center}
		$$
		\begin{aligned}
			& \max_{\Omega}w=w\rvert_{\partial D}=1=G(R),\\
			& w_m=\min_{\Omega}w=w\rvert_{\partial\Omega}=G(R+\delta)=v_m,\\
			&\abs{Dw}_{w=t}=\abs{Dv}_{v=t}=g(t),\;\;w_m\leq t\leq1.
		\end{aligned}
		$$
	\end{center}
Then $$I_{\beta,p,\delta}(D)\leq\int_{\Omega\setminus D}{\abs{Dw}}^p \;dx+\beta \int_{\partial\Omega}{\abs{w}}^p\;d\mathcal{H}^{n-1}.$$
Let $$E_t=\{x\in \Omega:w(x)>t\}=\{x\in \Omega:d(x)<G^{-1}(t)\}=D+G^{-1}(t)B$$
and 
$$B_t=\{x\in\Omega_*:v(x)>t\}.$$
Being $W_{n-1}(D)=W_{n-1}(D^*)$, using the Steiner formula and (\ref{quer}), we get for $w_m<t\leq 1$ and $\rho=G^{-1}(t)$ that
\begin{center}
	$$\begin{aligned}
		P(E_t)=P(D+\rho B)&=n \sum_{n=0}^{n-1}\binom{n-1}{k}W_{k+1}(D)\rho^k\\
		& \leq \sum_{n=0}^{n-1}\binom{n-1}{k}W_{k+1}(D^*)\rho^k=P(D^*+\rho B)=P(B_t).\\
	\end{aligned}
$$
\end{center}
Hence $$\int_{w=t} \abs{Dw}d\mathcal{H}^{n-1}=g(t)P(E_t)\leq g(t)P(B_t)=\int_{v=t}\abs{Dv}d\mathcal{H}^{n-1}, \qquad w_m<t\leq 1$$
then, by co-area formula 
\begin{center}
	$$\begin{aligned}
		\int_{\Omega\setminus D}{\abs{Dw}}^{p}dx
		&=\int_{w_m}^{1}dt\int_{w=t}{\abs{Dw}}^{p-1}d\mathcal{H}^{n-1}\\
		&=\int_{w_m}^{1}[g(t)]^{p-1}P(E_t)dt\leq\int_{w_m}^{1}[g(t)]^{p-1}P(B_t)dt\\
		&=\int_{\Omega\setminus D_*}{\abs{Dv}}^p dx.
	\end{aligned}
	$$
\end{center}
Since by construction $w=w_m=v_m$ on $\partial \Omega$ and $P(\Omega)\leq P(\Omega_*)$, we have
$$ \int_{\partial\Omega} {\abs{w}}^p d\mathcal{H}^ {n-1}={\abs{w_m}}^{p}P(\Omega)={\abs{v_m}}^pP(\Omega_*)=\int_{\partial \Omega_*}{\abs{v}}^p d\mathcal{H}^{n-1}.$$
So,
\begin{center}
	$$
	\begin{aligned}
		I_{\beta,p, \delta}(D)&\leq\int_{\Omega\setminus D}{\abs{Dw}}^pdx+\beta\int_{\partial \Omega} {\abs{w}}^pd\mathcal{H}^{n-1}\\
		&\leq \int_{\Omega_*\setminus D^*}{\abs{Dv}}^pdx+\beta\int_{\partial \Omega_*}{\abs{v}}^pd\mathcal{H}^{n-1}=I_{\beta,p, \delta}(D^*).
	\end{aligned}
	$$
\end{center} 

\end{proof}

\section{Remarks}
\label{sec6}
There is a counterintuitive behaviour of the functional $I_{\beta,p,\delta}(D,\Omega)$ when $\Omega$ and $D$ are concentric balls, which is very peculiar and that can be summarized in next two propositions.
\begin{prop}
	Let $B_R$ be a ball of radius $R$. If $\beta\geq\biggl[\dfrac{n-1}{R(p-1)}\biggl]^{p-1}$ then $I_{\beta,p, \delta}(B_R)$ is decreasing in $\delta$. When $\beta<\biggl[\dfrac{n-1}{R(p-1)}\biggl]^{p-1}$, then $I_{\beta,p,\delta}(B_R)$ is increasing for $\delta<\biggl(\dfrac{n-1}{p-1}\biggl)\dfrac{1}{{\beta}^{\frac{1}{p-1}}}-R$ and decreasing for $\delta>\biggl(\dfrac{n-1}{p-1}\biggl)\dfrac{1}{{\beta}^{\frac{1}{p-1}}}-R$.
\end{prop}
\begin{proof}
	If $D=B_R$ is a ball with radius $R$, then obviously $\Omega=B_R+\delta B=B_{R+\delta}$ and the minimum of (\ref{funzionale}), $u(x)=u(r)$, is 
	\begin{equation}
		\label{sol}
			u(r)=
		\begin{cases}
			1-{\gamma_1}^{\frac{1}{p-1}}\biggl(\dfrac{p-1}{p-n}\biggl)R^{\frac{p-n}{p-1}}\biggl[\biggl(\dfrac{r}{R}\biggl)^{\frac{p-n}{p-1}}-1\biggr] \qquad & p\not =n\\
				1-{\gamma_1}^{\frac{1}{n-1}}\log {\frac{r}{R}} \qquad &p=n
		\end{cases}	
	\end{equation}
for a suitable constant $\gamma_1$. This follows from the fact that 
$$r^{n-1}\Delta_p u= \frac{d}{dr}(r^{n-1}\abs{u'(r)}^{p-2}u'(r))=0,$$
but $u$ is decreasing and positive, then
$$u'(r)=-\dfrac{{\gamma_1}^{\frac{1}{p-1}}}{{r}^{\frac{n-1}{p-1}}}, \qquad r \in [R,R+\delta]$$
If $p \not = n$, then integrating by parts and keeping in mind that $u(R)=1$ it holds

$$u(r)=-{\gamma_1}^{\frac{1}{p-1}}\biggl(\dfrac{p-1}{p-n}\biggl)\biggl[r^{\frac{p-n}{p-1}}-R^{\frac{p-n}{p-1}}\biggr]+1.$$
If $p=n$,
$$u(r)=-{\gamma_1}^{\frac{1}{n-1}}\log \dfrac{r}{R} + 1.$$
Now, we can find $\gamma_1$ using the boundary conditions.\\
If $p\not =n$,then using Robin condition on $\partial\Omega$ it holds that
$$-(-u'(R+\delta))^{p-1}+\beta(u(R+\delta))^{p-1}=0.$$
By using the explicit expression of $u$,
$${\gamma_1}^{\frac{1}{p-1}}(R+\delta)^{-\frac{n-1}{p-1}}=\beta^{\frac{1}{p-1}}\biggl(-{\gamma_1}^{\frac{1}{p-1}}\biggl(\dfrac{p-1}{p-n}\biggl) R^{\frac{p-n}{p-1}}\biggl(1-\biggl(\dfrac{R+\delta}{R}\biggl)^{\frac{p-n}{p-1}}\biggl)+1\biggl)$$
and we have that
$$\gamma_1=\dfrac{\beta}{\biggl[(R+\delta)^{-\frac{n-1}{p-1}}+\biggl(\dfrac{p-1}{p-n}\biggl)\beta^{\frac{1}{p-1}}R^{\frac{p-n}{p-1}}\biggl(\biggl(\dfrac{R+\delta}{R}-1\biggl)^{\frac{p-n}{p-1}}\biggl)\biggl]^{p-1}}.$$
In particular, 
$$\int_{\partial (\Omega\setminus D)}  ({\abs{Du}}^{p-2}Du \cdot \nu)\;d\mathcal{H}^{n-1}=0$$
and using the divergence theorem
$$-\int_{\Omega\setminus D}{\abs{Du}}^{p}+\int_{\partial (\Omega \setminus D)}u({\abs{Du}}^{p-2}Du \cdot \nu)=0$$
$$-\int_{\Omega\setminus D}{\abs{Du}}^{p}-\int_{\partial \Omega} \beta u^p+\int_{\partial D} u({\abs{Du}}^{p-2}Du \cdot \nu)=0.$$
This implies that
$$I_{\beta,p,\delta}(B_R)=\int_{\partial B_R} {\abs{Du}}^{p-2}\dfrac{\partial u}{\partial \nu}\;d\mathcal{H}^{n-1}.$$
Let us observe that 
$$I_{\beta,p,\delta}(B_R)=n\omega_n\gamma_1$$
and we have 
 $$\partial_\delta[I_{\beta,p,\delta}(B_R)]<0$$
 if
 $$\partial_\delta \biggl[(R+\delta)^{\frac{1-n}{p-1}}+\beta^{\frac{1}{p-1}}\biggl(\dfrac{p-1}{p-n}\biggl)R^{\frac{p-n}{p-1}}\biggl(\biggl(\dfrac{R+\delta}{R}\biggl)^{\frac{p-n}{p-1}}-1\biggl)\biggl]>0$$
and this is true if
$$\delta>\biggl(\dfrac{n-1}{p-1}\biggl)\dfrac{1}{{\beta}^{\frac{1}{p-1}}}-R.$$
On the other hand, when $p=n$ we have
$$\gamma_1=\dfrac{\beta}{\biggl[\frac{1}{(R+\delta)}+{\beta}^{\frac{1}{n-1}}\log\frac{R+\delta}{R}\biggl]^{n-1}}$$
and 
$$I_{\beta,n,\delta}(B_R)=\dfrac{n\omega_n\beta}{{\biggl[\frac{1}{(R+\delta)}+{\beta}^{\frac{1}{n-1}}\log\frac{R+\delta}{R}\biggl]^{n-1}}}.$$
So
$$\partial_\delta [I_{\beta,p,\delta}(B_R)]<0$$
if
$$\partial_\delta\biggl[\dfrac{1}{(R+\delta)}+\beta^{\frac{1}{n-1}}\log\biggl(1+\frac{\delta}{R}\biggl)\biggl]>0$$
and this is true if and only if
$$\delta > \dfrac{1}{\beta^{\frac{1}{n-1}}}-R,$$
and the proposition is proved.
\end{proof}
To show next result, we first need the following Lemma, using the following notation.\\ Let $D=B_R \subset \Omega$, and denote by 
$$
\begin{aligned}
	P&=P(B_R)=n\omega_nR^{n-1},\;V=\abs{B_R}\\
	&=\omega_nR^n,\;\Delta P=P(\Omega)-P(B_R),\;\Delta V=\abs{\Omega}-\abs{B_R}.
\end{aligned}
$$
\begin{lemma}
	\label{cost}
	Let $D=B_R$ and $B_R \subset \Omega$. For any $\delta_0>0$, there exists a constant $$C=\dfrac{n\omega_n R^{n-1}}{\delta_0}\biggl[\biggl(1+\dfrac{\delta_0}{\omega_n R^n}\biggl)^{1-\frac{1}{n}}-1\biggl]$$
	such that if $\Delta V\leq \delta_0$ it holds that $$\Delta P \geq C\Delta V.$$
\end{lemma}
We refer to \cite{prof} for the proof. \\ Now, we want to prove that, in the regime $\beta$ "small", if the thickness $\delta$ is below a certain threshold value, $I_{\beta,p}(B_R,\Omega)$ is greater then $I_{\beta,p}(B_R,B_R)$.
\begin{prop}
	Let $D=B_R(0)$ and $\beta<\biggl[\dfrac{n-1}{R(p-1)}\biggl]^{p-1}$. Then there exists a positive constant $\delta_0$ such that for any bounded domain $\Omega$, with $D\subset \Omega$ and $\abs{\Omega}-\abs{D}<\delta_0$, then $$I_{\beta,p}(B_R,\Omega)>I_{\beta,p}(B_R,B_R)$$
\end{prop}
\begin{proof}
	Let $u$ be the minimizer of $I_{\beta,p}(B_R,\Omega)$. Consider \\
	$$\Sigma=\Omega \setminus B_R, \qquad\Gamma_m=\partial \Omega \setminus \partial B_R, \qquad \Gamma_t=\partial \{u>t\}\setminus\partial B_R, \qquad \Gamma_1=\partial B_R \cap \Omega$$\\
	and
	$$p(t)=P(\{u>t\}\cap \Sigma), \qquad \textrm{for a.e.}\;t>0.$$
	We want to show that 
	$$I_{\beta,p}(B_R;B_R)=\beta P(B_R)< I_{\beta,p}(B_R;\Omega)=\int_\Omega {\abs{Du}}^p dx+\beta \int_{\partial\Omega} {\abs{u}}^p d\mathcal{H}^{n-1}$$
	or equivalently $$\mathcal{H}^{n-1}(\Gamma_1)<\frac{1}{\beta}\int_{\Omega}{\abs{Du}}^p dx+\int_{\Gamma_0} {\abs{u}}^p d\mathcal{H}^{n-1}$$
then using coarea formula and Fubini theorem we have 

\begin{equation}
	\label{eq1}
		\begin{aligned}
		\int_{0}^{1} t^{p-1}p(t)dt&=\int_{0}^{1}       t^{p-1}P(\{u>t\}\cap \Sigma)dt\\\\
		&= \int_{0}^{1}\biggl(\int_{\Gamma_1} t^{p-1} d\mathcal{H}^{n-1}\biggl)dt+\int_{0}^{1}\biggl(\int_{\Gamma_t \cap \Omega} t^{p-1} d\mathcal{H}^{n-1}\biggl)dt+ \int_{0}^{1}\biggl(\int_{\Gamma_t \cap \partial \Omega} t^{p-1} d\mathcal{H}^{n-1}\biggl)dt\\\\
	&=\dfrac{\mathcal{H}^{n-1}(\Gamma_1)}{p}+ \dfrac{1}{p} \int_{\Gamma_0}u^p d\mathcal{H}^{n-1}+ \int_{\Omega} u^{p-1}\abs{Du}dx.
	\end{aligned}
\end{equation}

	 From the Lemma \ref{cost} we know that for $\abs{\Sigma}<\delta_0$, with $\delta_0$ fixed,
     $$p(t)-2\mathcal{H}^{n-1}(\Gamma_1)\geq C \mu (t)$$
	 and then
	 $$ \int_{0}^{1}t^{p-1}p(t)dt \geq 2 \int_{0}^{1}t^{p-1}\mathcal{H}^{n-1}(\Gamma_1)dt+ C \int_{0}^{1}t^{p-1}\mu(t)dt=\dfrac{2}{p}\mathcal{H}^{n-1}(\Gamma_1)+\dfrac{C}{p}\int_{\Omega} u^p dx,$$
	 where $C$ is the constant of the Lemma \ref{cost}. 
	 Hence, substituing in (\ref{eq1}) 
	 $$\dfrac{2}{p}\mathcal{H}^{n-1}(\Gamma_1)+\dfrac{C}{p}\int_{\Omega}u^p dx \leq \dfrac{\mathcal{H}^{n-1}(\Gamma_1)}{p}+ \dfrac{1}{p}\int_{\Gamma_0}u^p d\mathcal{H}^{n-1}+\int_{\Omega}u^{p-1}\abs{Du}dx.$$
	 On the other hand, by the Young inequality
	 $$\int_{\Omega} u^{p-1}\abs{Du}dx \leq \dfrac{p-1}{p\epsilon^{\frac{1}{p-1}}}\int_{\Omega}u^pdx+\dfrac{\epsilon}{p}\int_{\Omega} {\abs{Du}}^pdx$$
	 $$\mathcal{H}^{n-1}(\Gamma_1)+ C\int_{\Omega}{\abs{u}}^pdx \leq \int_{\Gamma_0}u^p d\mathcal{H}^{n-1}+\dfrac{p-1}{\epsilon^{\frac{1}{p-1}}}\int_{\Omega} u^pdx+ \epsilon \int_{\Omega}{\abs{Du}}^pdx=$$
	 $$=\int_{\Gamma_0} u^p d\mathcal{H}^{n-1}+\dfrac{p}{\epsilon^{\frac{1}{p-1}}}\int_{\Omega}u^pdx-\dfrac{1}{\epsilon^{\frac{1}{p-1}}}\int_{\Omega} u^pdx+\epsilon \int_{\Omega} {\abs{Du}}^pdx$$
	 and choosing $\epsilon=\dfrac{1}{\beta}$ it holds that
	 $$\mathcal{H}^{n-1}(\Gamma_1)+[C-\beta^{\frac{1}{p-1}}(p-1)]\int_{\Omega}u^pdx \leq \int_{\Gamma_0} u^p d\mathcal{H}^{n-1}+\dfrac{1}{\beta}\int_{\Omega} {\abs{Du}}^pdx.$$
	 Then, being $R<\dfrac{n-1}{(p-1)\beta^{\frac{1}{p-1}}}$, for $\delta_0$ sufficiently small the constant $C$ is larger then $\beta^{\frac{1}{p-1}}(p-1)$ and thesis follows.
\end{proof}

\section*{Acknowledgments}
This work has been partially supported by GNAMPA of INdAM.
\addcontentsline{toc}{chapter}{ Bibliografia}

\bibliographystyle{Abbrv}
\bibliography{bibliografia}

\end{document}